\documentclass[reqno0,fleqn,12pt]{amsart}
\pdfoutput=1
\usepackage{dsfont}
\usepackage{bbm}
\usepackage[nodate]{datetime}
\usepackage[hmargin=30mm,top=30mm,bottom=30mm,a4paper]{geometry}
\usepackage{color}
\usepackage{subfig}
\usepackage{fancyhdr}
\usepackage{amsmath,amssymb,amsfonts,amsthm}
\usepackage{amstext}
\usepackage{amsmath}
\usepackage{amssymb}
\usepackage{enumerate}
\usepackage{amsbsy}
\usepackage{amsopn}
\usepackage{bbm,amsthm}
\usepackage{amscd}
\usepackage{amsxtra}
\usepackage{todonotes}
\usepackage{hyperref}

\newtheorem{theorem}{Theorem}[section]
\newtheorem*{theorem*}{Theorem}
\newtheorem{lemma}{Lemma}[section]

\theoremstyle{remark}

\newcommand{\CC}{\mathds{C}}

\newcommand{\RR}{\mathds{R}}
\newcommand{\NN}{\mathds{N}}

\newcommand{\EE}{\mathbb{E}}
\newcommand{\PP}{\mathbb{P}}

\begin{document}
\title{Law of the iterated logarithm for a random Dirichlet series}
\author{Marco Aymone, Susana Fr\'ometa, Ricardo Misturini}
\begin{abstract}
Let $(X_n)_{n\in \mathds N}$ be a  sequence of \textit{i.i.d.} random variables with distribution $\mathbb P(X_1=1)=\mathbb P(X_1=-1)=1/2$. Let $F(\sigma)=\sum_{n=1}^\infty X_nn^{-\sigma}$. We prove that the following holds almost surely
\begin{equation*}
\limsup_{\sigma\to 1/2^+}\frac{F(\sigma)}{\sqrt{2\mathbb E F(\sigma)^2\log\log \mathbb E F(\sigma)^2}}=1.
\end{equation*}
\end{abstract}

\maketitle

\section{Introduction.}  Let $\mathcal{S}=\{1\leq n_1<n_2<\ldots\}$ be a set of non-negative real numbers and $(a_n)_{n\in\mathcal{S}}$ be a sequence of complex numbers. A Dirichlet series is a series of the form $F(s)=\sum_{n\in\mathcal{S}} a_n n^{-s}$, where $s$ is a complex number $s=\sigma+it$. A standard result for series of this type is that if $F$ converges at $s=s_0$, then it converges at all $s\in\CC$ with $Re(s)>Re(s_0)$, and $F$ defines an analytic function in the half plane $\{s\in\CC:Re(s)>Re(s_0)\}$. Hence, when $F$ converges at some point $s_0$, the following abscissa of convergence is well defined: $\sigma_c:=\inf\{\sigma\in\RR: F(\sigma)\mbox{ converges }\}$. 

An important example of a Dirichlet series is the Riemann $\zeta$ function:
\begin{equation*}
\zeta(\sigma):=\sum_{n=1}^\infty \frac{1}{n^\sigma}.
\end{equation*} 
It follows that $\zeta (\sigma)$ has abscissa of convergence $\sigma_c=1$. Moreover, $\zeta(\sigma)$ has a singularity at $s=1$. Indeed, as $\sigma\to 1^+$, $\zeta(\sigma)\sim\frac{1}{\sigma-1}$.

The study of the behavior of a Dirichlet series near its line of abscissa of convergence $\sigma_c$ is classical in Analysis and in Analytic Number Theory. For instance, one can obtain the prime number Theorem -- the statement that the number of primes below $x$ is asymptotically $x/\log x$  -- from the classical Wiener-Ikehara Theorem, a Tauberian result; see, for instance Chapter II.7 of \cite{tenenbaumlivro}.

Let $(X_n)_{n\in\NN}$ be  \textit{i.i.d.} random variables with $\PP(X_1=1)=\PP(X_1=-1)=1/2$. In this paper we are interested in the behavior of the random Dirichlet series
\begin{equation}\label{definition of F}
F(\sigma):=\sum_{n=1}^\infty \frac{X_n}{n^\sigma}   
\end{equation}
near its abscissa of convergence $\sigma_c$. By the Kolmogorov's one-series Theorem, $F(\sigma)$ converges if and only if $\sigma>1/2$, and thus $\sigma_c=1/2$. 

We say that a Dirichlet series is analytic in its abscissa of convergence if this Dirichlet series has an analytic continuation to the open set consisted of the union of the half plane $Re(s)>\sigma_c$ with an open ball with some positive radius and centered at $\sigma_c$. It is important to observe that if such analytic continuation exists, then it is unique. In this terminology, sometimes a Dirichlet series may be analytic in its abscissa of convergence $\sigma_c$, for example, the Dirichlet $\eta$ function $\eta(\sigma)=\sum_{n=1}^\infty (-1)^{n+1}n^{-\sigma}$ that has $\sigma_c=0$. Indeed, the Riemann $\zeta$ function  has analytic continuation to $\CC\setminus\{1\}$ with a simple pole at $s=1$, and for $s\neq 1$ in the half  plane $Re(s)>0$ we have the formula $\eta(s)=(1-2^{1-s})\zeta(s)$. Since $(1-2^{1-s})$ is an entire function and has a zero at $s=1$, we obtain that $\eta(s)$ has analytic continuation to $\CC$, in particular it is analytic in an open set containing its abscissa of convergence. On the other hand, sometimes a Dirichlet series has a singularity in its abscissa of convergence, which is, for instance the case of our Random Dirichlet series $F(\sigma)$; see, for instance, Theorem 4, pg. 44 of the book of Kahane \cite{kahane}.

In \cite{aymonerealzerosofdirichletseries}, it has been shown that, with probability 1, the function $F$ has infinitely many zeroes accumulating at $1/2$. To prove that, the following Central Limit Theorem has been established: $F(\sigma)/\sqrt{\EE F(\sigma)^2}\to_d\mathcal{N}(0,1)$, as $\sigma\to1/2^+$, where $\mathcal{N}(0,1)$ stands for the standard Gaussian distribution. Moreover, it has been proved that,  \textit{almost surely},
\begin{equation*}
\limsup_{\sigma\to 1/2^+}\frac{F(\sigma)}{\sqrt{\EE F(\sigma)^2}} =\infty.
\end{equation*}
Thus, a natural question is what is the asymptotics of $F(\sigma)$ as $\sigma\to1/2^+$. Our main result states:

\begin{theorem}\label{teorema principal} Let $F(\sigma)$ be the random Dirichlet series defined in \eqref{definition of F}. Then
\begin{equation*}
\limsup_{\sigma\to 1/2^+}\frac{F(\sigma)}{\sqrt{2\EE F(\sigma)^2\log \log \EE F(\sigma)^2}}=1, \mbox{ almost surely}.
\end{equation*}
\end{theorem}
Since $F(\sigma)$ is a symmetric random variable, we have the $\liminf$ of the same quantity above equals to $-1$.

As $\sigma\to1/2^+$, $\EE F(\sigma)^2/(2\sigma-1)^{-1}\to 1$ (see Lemma \ref{lemma riemann sums}), hence Theorem \ref{teorema principal} is equivalent to: Almost surely
\begin{equation*}
\limsup_{\sigma\to 1/2^+}\frac{F(\sigma)}{\sqrt{\frac{2}{2\sigma-1}\log \log \frac{1}{2\sigma-1}}}=1.
\end{equation*} 
 
 Theorem \ref{teorema principal} is the corresponding Law of the Iterated Logarithm (LIL) for the random Dirichlet series $F(\sigma)$. For the random geometric series, $G(\beta):=\sum_{n=0}^\infty X_n\beta^n$, studied by Bovier and Picco in  \cite{bovierepicolimittheorems} and \cite{bovierrandomgeometric}, it has been established that, almost surely
\begin{equation*} 
\limsup_{\beta\to 1^-} \frac{G(\beta)}{\sqrt{2\EE G(\beta)^2 \log\log \EE G(\beta)^2}}=1.
\end{equation*}

 The main issue to obtain these results is that, in comparison with the classical LIL for the simple random walk, we do not have at our disposal a similar result to the Levy's maximal inequality:
 \begin{equation}\label{equation Levy maximal}
\PP\bigg{(}\max_{1\leq m \leq n}\bigg{|}\sum_{k=1}^m X_k \bigg{|}\geq t \bigg{)}\leq 3 \max_{1\leq m \leq n}\PP\bigg{(}\bigg{|}\sum_{k=1}^m X_k \bigg{|}\geq \frac{t}{3} \bigg{)}.
\end{equation}

In the classical proof of the LIL for $S(x)=\sum_{n\leq x}X_n$, the size of $S(x_k)$ is controlled along a sequence $x_k\to\infty$, and the size of $S(x)$ for $x\in[x_k,x_{k+1}]$ is controlled via (\ref{equation Levy maximal}). In our case and in the random geometric case, the supremum is taken over continuous parameters and a maximal inequality is not available. 

The proof of Theorem \ref{teorema principal} is divided into two main steps: an upper bound and a lower bound. For the lower bound we follow the ideas of \cite{bovierepicolimittheorems} to show that for any $\gamma>0$, there is a sequence $\sigma_k\to 1/2^+$ such that, \textit{almost surely}, 
\begin{equation}\label{equacao cota inferior} 
\limsup_{k\to\infty} \frac{F(\sigma_k)}{\sqrt{\EE F(\sigma_k)^2\log\log \EE F(\sigma_k)^2}}\geq 1-\gamma.
 \end{equation}
To show that, one main ingredient is to find a lower bound for 
\begin{equation}\label{intro lower}
    \PP\left(\frac{F(\sigma_k)}{\sqrt{\EE F(\sigma_k)^2\log\log \EE F(\sigma_k)^2}}\geq 1-\gamma\right)
\end{equation}
using standard large deviation techniques, and this is made in Lemma \ref{lowerbound}. We conclude the proof of the lower bound using the second Borel-Cantelli lemma, and for that, we will construct independent events that are asymptotic equivalent to those in \eqref{intro lower}, as $k\to\infty$.

For the upper bound, we show that over an specific sequence $\sigma_k\to 1/2^+$, 
\begin{equation}\label{equacao cota superior} 
\limsup_{k\to\infty} \frac{F(\sigma_k)}{\sqrt{\EE F(\sigma_k)^2\log\log \EE F(\sigma_k)^2}}\leq 1+\gamma.
 \end{equation}
 Then we control the the size of $F(\sigma)$ for $\sigma\in[\sigma_k,\sigma_{k-1}]$ by following an approach different from the one in \cite{bovierepicolimittheorems}, where it was used a renormalization idea that is suitable for geometric series. Here we argue as in the proof of the Kolmogorov-\v{C}entsov Theorem; see, for instance, Chapter 2.2 of \cite{karatzaslivro}. Indeed, we consider a dyadic partition of each interval $[\sigma_k,\sigma_{k-1}]$, that is, intervals of the form $[\tau_{l,n}(k),\tau_{l,n+1}(k)]$ where $\tau_{l,n}(k)=\sigma_k+\frac{n}{2^l}(\sigma_{k-1}-\sigma_k)$. Then we exploit the fact that $F(\sigma)$ is differentiable as a function of $\sigma$, and, with that, we control the size of the difference of $F$ at  consecutive elements of the dyadic partition: $|F(\tau_{l,n}(k))-F(\tau_{l,n+1}(k))|$. 
 
 Here we present some heuristics that will give us the intuition of the bound that will be obtained in Lemma \ref{lemma controle entre os intervalos}. We have by the mean value theorem that 
 \begin{equation*}
 |F(s)-F(t)|\leq |s-t|\max_{u\in[s,t]}|F'(u)|,
 \end{equation*}
  and this inequality is nearly optimal if $F'$ is continuous and $s$ and $t$ are close to each other. On the one hand, the derivative of a Dirichlet series is an analytic function, since it is also a Dirichlet series with same abscissa $\sigma_c$: $F'(\sigma)=-\sum_{n=1}^\infty X_n n^{-\sigma}\log n$. On the other hand, by standard estimates, for $\sigma$ close to $1/2^+$, $\EE F'(\sigma)^2 =\sum_{n=1}^\infty n^{-2\sigma}(\log n)^2\sim \frac{1}{(2\sigma-1)^3}$. Then we show that, if $s,t\in[\sigma_k,\sigma_{k-1}]$, $|F(s)-F(t)|$ is bounded above by something that behaves as 
\begin{align*}
\sqrt{\EE |F(\sigma_k)-F(\sigma_{k-1})|^2}&\leq |\sigma_k-\sigma_{k-1}|\max_{u\in[\sigma_k,\sigma_{k-1}]}\sqrt{\EE F'(u)^2}\\
&\ll \frac{|\sigma_k-\sigma_{k-1}|}{(2\sigma_k-1)^{3/2}}\\
&\ll\frac{|2\sigma_k-1|}{(2\sigma_k-1)^{3/2}}\\
&= \frac{1}{(2\sigma_k-1)^{1/2}}\\
&\leq\sqrt{\EE F(\sigma_k)^2},
\end{align*}
where, in the third line above it is used a particular property of the chosen sequence~$\sigma_k$. Combining this with (\ref{equacao cota superior}), we obtain the upper bound
\begin{equation*}
\limsup_{\sigma\to1/2^+} \frac{F(\sigma)}{\sqrt{\EE F(\sigma)^2\log\log \EE F(\sigma)^2}}\leq 1+\gamma.
 \end{equation*}

\section{Preliminaries}
\subsection{Notation} Here we use $f(x)\ll g(x)$ whenever there exists a constant $c>0$ such that $|f(x)|\leq c |g(x)|$, in a certain range of $x$ -- This range could be all the interval $x\in[0,\infty)$ or $x\in (a-\delta,a+\delta)$, $a\in \RR,\delta>0$. We say that $f(x)\sim g(x)$ if $\lim\frac{f(x)}{g(x)}=1$. 

Here, $F(\sigma)=\sum_{n=1}^\infty \frac{X_n}{n^\sigma}$, where $X_n$ are \textit{i.i.d.} random variables with $\PP(X_1=1)=\PP(X_1=-1)=1/2$. By the Kolmogorov's one-series Theorem, it follows that $F(\sigma)$ is convergent for all $\sigma>1/2$ and divergent for $\sigma\leq 1/2$. Moreover, for $s=\sigma+it$, in the half plane $Re(s)>1/2$, $F(s)$ is an analytic function; see Chapter I of \cite{montgomerylivro}.

\subsection{Estimates for the Riemann $\zeta$ function}
We begin with some standard estimates for the Riemann $\zeta$ function. These are classical, and we provide a proof here for the convenience of the reader.
\begin{lemma}\label{lemma riemann sums} Let $\sigma>1$. As $\sigma\to1$, $\zeta(\sigma)$ is of the order of $\frac{1}{\sigma-1}$, in fact we have that
\begin{equation*}
\frac{1}{\sigma-1}\leq \zeta(\sigma)\leq \frac{\sigma}{\sigma-1}.    
\end{equation*}
Moreover, for any $M>1$  
\begin{align*}
&\sum_{n=1}^M \frac{1}{n^\sigma}\leq \frac{1}{\sigma-1}\bigg{(}\sigma-\frac{1}{M^{\sigma-1}} \bigg{)}\\
&\sum_{n>M}\frac{1}{n^\sigma}\leq \frac{1}{(\sigma-1)M^{\sigma-1}}.
\end{align*}
\end{lemma}

\begin{proof}
Since the function $f(t)=1/t^\sigma$ is decreasing for $t>0$, we can compare the sum with the integral obtaining
\begin{equation*}
\int_1^{M+1}\frac{1}{t^\sigma}dt\leq \sum_{n=1}^M\frac{1}{n^\sigma}\leq1+\int_1^M\frac{1}{t^\sigma}dt \quad \text{and} \quad \sum_{n=M+1}^\infty \frac{1}{n^\sigma} \leq \int_{M}^\infty\frac{1}{t^\sigma}dt,
\end{equation*}
which gives the desired estimates.
\end{proof}

%\begin{proof}
%Let $[x]$ and $\{x\}$ be the integer and the fractional part of $x$, respectively. Thus, for $\sigma>1$
%\begin{align*}
%\zeta(\sigma)&=\int_{1^-}^\infty t^{-\sigma}d[t]=\int_{1^-}^\infty t^{-\sigma}d(t-\{t\})=\int_{1^-}^\infty t^{-\sigma}dt-\int_{1^-}^\infty t^{-\sigma}d\{t\}\\
%&\frac{1}{\sigma-1}+1-\sigma\int_1^\infty t^{-(\sigma+1)}\{t\}dt,
%\end{align*}
%where in the last step we used integration by parts. Then, $\zeta(\sigma)\leq \frac{1}{\sigma-1}+1=\frac{\sigma}{\sigma-1}$. On the other hand, as $\{t\}\leq 1$,
%\begin{align*}
%-\sigma\int_1^\infty t^{-(\sigma+1)}\{t\}dt\geq -\sigma\int_1^\infty t^{-(\sigma+1)}=-1,
%\end{align*}
%and hence $\zeta(\sigma)\geq \frac{1}{\sigma-1}$.

%For the second part of the Lemma, replace $\infty$ by $M$ in the calculation above. This yelds:
%\begin{align*}
%\sum_{n=1}^M\frac{1}{n^{\sigma}}&=\int_{1}^M t^{-\sigma}dt-\int_{1^-}^M t^{-\sigma}d\{t\}\\
%&=\frac{1}{\sigma-1}\bigg{(}1-\frac{1}{M^{\sigma-1}} \bigg{)}+1-\frac{\{M\}}{M^\sigma}-\sigma\int_1^M t^{-(\sigma+1)}\{t\}dt \\ 
%&\leq  \frac{1}{\sigma-1}\bigg{(}\sigma-\frac{1}{M^{\sigma-1}} \bigg{)} .
%\end{align*}
%Finally, since the function $f(t)=t^{-\sigma}$ is decreasing, we obtain that
%\begin{align*}
%\sum_{n>M}\frac{1}{n^\sigma}\leq\int_M^\infty t^{-\sigma}dt=\frac{1}{(\sigma-1)M^{\sigma-1}}.
%\end{align*}
%\end{proof}

\subsection{Some basic results for $\sum_{k=1}^\infty a_kX_k$}

\begin{lemma}\label{infiteproduct}
	Let $\{X_k\}_{k\geq 1}$ be a sequence of \textit{i.i.d.} random variables with $\PP (X_1=1)=\PP (X_1=-1)=1/2$, and $\{a_k\}_{k\geq 1}$ a sequence of real numbers such that $\sum_{k=1}^\infty a_k^2<\infty$, then
	\begin{equation*}
	\EE \left[\exp\left(\sum_{k=1}^\infty a_kX_k\right)\right]=\prod_{k=1}^{\infty}\EE \left[\exp(a_kX_k)\right]\leq\exp\left(\frac{1}{2}\sum_{k=1}^\infty a_k^2\right)<\infty.
	\end{equation*} 
\end{lemma}
\begin{proof}
   Notice that, since $\log\cosh x\leq\frac{x^2}{2}$, we have
   \begin{equation*}
   \prod_{k=1}^{\infty}\EE \left[\exp(a_kX_k)\right]=\exp\left(\sum_{k=1}^\infty\log\cosh a_k\right)\leq\exp\left(\frac{1}{2}\sum_{k=1}^\infty a_k^2\right)<\infty.
   \end{equation*}
   The inequality 
	\begin{equation*}
	\EE \left[\exp\left(\sum_{k=1}^\infty a_kX_k\right)\right]\leq\prod_{k=1}^{\infty}\EE \left[\exp(a_kX_k)\right]
	\end{equation*}	
    follows from Fatou's Lemma. In order to prove the equality, let us define $Y_n=\prod_{k=1}^n e^{a_kX_k}$. We want to use the dominated convergence theorem to show that $\EE [\lim_{n\to\infty}Y_n]=\lim_{n\to\infty}\EE [Y_n]$. Observe that $Y_n$ is a non-negative submartingale with respect to the $\sigma$-algebra $\mathcal{F}_n$ generated by $\left\{X_1,\ldots,X_n\right\}$, indeed
	\begin{equation*}
	\EE \left[Y_{n+1}|\mathcal{F}_n\right]=Y_n\EE \left[e^{a_{n+1}X_{n+1}}\right]=Y_n\cosh a_{n+1}\geq Y_n.
	\end{equation*} 
	Also notice that
	\begin{equation*}
	\EE\left[Y_n^2\right]= \exp\left(\sum_{k=1}^n \log\cosh 2a_k\right)\leq \exp\left(2\sum_{k=1}^n a_k^2\right)\leq \exp\left(2\sum_{k=1}^\infty a_k^2\right)<\infty.
	\end{equation*}
	Using Cauchy-Schwarz and then Doob's inequality, we obtain
    \begin{equation*}
    \EE \left[\max_{1\leq k\leq n}Y_k\right]\leq \EE\left[\max_{1\leq k\leq n}Y_n^2\right]^{1/2}\leq 2\EE \left[Y_n^2\right]^{1/2} \leq 2\exp\left(\sum_{k=1}^\infty a_k^2\right) <\infty.
    \end{equation*}
    Then, by Fatou's Lemma, $\EE\left[\sup_{k\geq 1}Y_k\right]<\infty$. Therefore, the proof is concluded using the dominated convergence theorem.
\end{proof}

In the following we will recall the Hoeffding's inequality. Since in some situations we will need this result for infinitely many summands, which holds in our case, we present the proof to make clear that such generalization is possible. The case of a finite number of summands is contained in the lemma below considering a sequence $\left\{a_k\right\}$ with only a finite number of non-zero terms. 

\begin{lemma}[Hoeffding's inequality]\label{hoeffding}
Let $\{X_k\}_{k\geq 1}$ be a sequence of \textit{i.i.d.} random variables with $\PP (X_1=1)=\PP (X_1=-1)=1/2$, and $\{a_k\}_{k\geq 1}$ a sequence of real numbers such that $\sum_{k=1}^\infty a_k^2<\infty$, then, for any $\lambda>0$,
\begin{equation*}
\PP\left(\sum_{k=1}^\infty a_kX_k\geq \lambda\right)\leq\exp\left(-\frac{\lambda^2}{2\sum_{k=1}^\infty a_k^2}\right).
\end{equation*} 
\end{lemma}
\begin{proof}
For $t\in \mathbb R$, by Markov's inequality and Lemma \ref{infiteproduct}, we have
\begin{align*}
\PP\left(\sum_{k=1}^\infty a_kX_k\geq \lambda\right)&=\PP\left(e^{t\sum_{k=1}^\infty a_kX_k}\geq e^{t\lambda}\right)
\leq e^{-t\lambda}\EE \left[e^{\sum_{k=1}^\infty ta_kX_k}\right]\\
&\leq \exp\left(-t\lambda+\frac{t^2}{2}\sum_{k=1}^\infty a_k^2\right).
\end{align*}	
Choosing $t=\lambda/\sum_{k=1}^\infty a_k^2$ we obtain the desired result.
\end{proof}

\section{Proof of the main result}
Let us adopt the notation 
\begin{equation}\label{barF}
\bar{F}(\sigma)=\frac{F(\sigma)}{\sqrt{\EE F(\sigma)^2}}.    
\end{equation}
The proof of Theorem \ref{teorema principal} will be made in four steps that we will describe in the following. 

\textbf{Step 1.} We first prove that, for all $\gamma>0$, there exists a deterministic sequence $\sigma_k\to \frac{1}{2}^+$ such that 
\begin{equation}\label{step1}
\PP \left(\limsup_{k\to\infty}\frac{\bar{F}(\sigma_k)}{\sqrt{2\log\log \EE F(\sigma_k)^2}}\geq 1-\gamma\right)=1.
\end{equation}

\textbf{Step 2.} Let $\epsilon>0$ be fixed and small. Then we prove that for the sequence $\sigma_k=\frac{1}{2}+\frac{1}{2\exp(k^{1-\delta})}$, with $0<\delta<\epsilon /2$, holds
\begin{equation}\label{step2}
\PP \left(\limsup_{k\to\infty}\frac{\bar{F}(\sigma_k)}{\sqrt{2\log\log \EE F(\sigma_k)^2}}\leq \sqrt{1+\epsilon}\right)=1.
\end{equation}

\textbf{Step 3.} Finally we prove that if $\sigma_k$ is as in the \textit{step 2}, then there exists a set $\Omega^*$ with probability $1$, such that for each $\omega\in\Omega^*$, there exists a $k_0=k_0(\omega)$, such that for all $k\geq k_0$, 
\begin{equation}\label{equacao controle dos intervalos}
\max_{\sigma\in[\sigma_k,\sigma_{k-1}]}|F(\sigma)-F(\sigma_k)|\ll \sqrt{\EE F (\sigma_k)^2}.
\end{equation}

\textbf{Step 4.} We conclude from (\ref{step2}) and (\ref{equacao controle dos intervalos}) that for any $\gamma>0$
\begin{equation}\label{equacao cota superior LIL}
\PP \left(\limsup_{\sigma\to1/2^+}\frac{\bar{F}(\sigma)}{\sqrt{2\log\log \EE F(\sigma)^2}}\leq 1+\gamma\right)=1,
\end{equation}
and hence, the Theorem \ref{teorema principal} follows from (\ref{step1}) and (\ref{equacao cota superior LIL}).

Now let us proceed to the execution of the steps described above.
\subsection*{Step 1}\label{subsection step1}
Let us split the normalized Dirichlet series $\bar F(\sigma)$ in three different parts: $F_1$, $F_2$ and $F_3$, where 
\begin{equation*}
 F_i(\sigma)=\frac{1}{\sqrt{\EE F(\sigma)^2}}\sum_{n=N_{i-1}+1}^{N_i}\frac{X_n}{n^\sigma},   
\end{equation*}
with $N_0=0$, $N_3=\infty$. The other parameters, $N_1=N_1(\sigma)$ and $N_2=N_2(\sigma)$, will be determined later in order to:
\begin{equation}\label{inicio}
	\displaystyle\PP \left(\limsup_{k\to\infty}\frac{|F_1(\sigma_k)|}{\sqrt{2\log\log \EE F(\sigma_k)^2}}=0\right)=1,
\end{equation}
\begin{equation}\label{cauda}
\displaystyle\PP \left(\limsup_{k\to\infty}\frac{|F_3(\sigma_k)|}{\sqrt{2\log\log \EE F(\sigma_k)^2}}=0\right)=1,
\end{equation}  
and
\begin{equation}\label{meio}
N_1(\sigma_{k+1})\geq N_2(\sigma_k).
\end{equation} 
The condition \eqref{meio} is required in order to $\{F_2(\sigma_k)\}_{k=1}^\infty$ be a family of independent random variables.  
   
 We use the first Borel-Cantelli lemma to prove \eqref{inicio} and \eqref{cauda} for suitable $\sigma_k$, $N_1(\sigma_k)$ and $N_2(\sigma_k)$. We would like to find sequences $\lambda_k$ and $\eta_k$ such that 
\begin{equation}\label{inicio1}
\sum_{k=1}^\infty\PP(|F_1(\sigma_k)|\geq\lambda_k)<\infty, \text{~~~~with~~~~} \frac{\lambda_k}{\sqrt{\log\log\EE F(\sigma_k)^2}}\to 0, 
\end{equation} 
and 
\begin{equation}\label{cauda1}
\sum_{k=1}^\infty\PP(|F_3(\sigma_k)|\geq\eta_k)<\infty, \text{~~~~with~~~~}
\frac{\eta_k}{\sqrt{\log\log\EE F(\sigma_k)^2}}\to 0.
\end{equation}
 
Using Lemma \ref{hoeffding}, we obtain the bounds
\begin{equation*}
\PP(|F_1(\sigma_k)|\geq\lambda_k)=2\PP(F_1(\sigma_k)\geq \lambda_k)\leq 2\exp\left(-\frac{\lambda_k^2\EE F(\sigma_k)^2}{2\sum_{n=1}^{N_1}n^{-2\sigma_k}}\right),
\end{equation*}
and
\begin{equation*}
\PP(|F_3(\sigma_k)|\geq\eta_k)=2\PP(F_3(\sigma_k)\geq \eta_k)\leq 2\exp\left(-\frac{\eta_k^2\EE F(\sigma_k)^2}{2\sum_{n=N_2+1}^{\infty}n^{-2\sigma_k}}\right).
\end{equation*}

Then, \eqref{inicio1} and \eqref{cauda1} will hold if we choose the sequences $\lambda_k=\sqrt{2(1+\epsilon)\alpha_k}$ and $\eta_k=\sqrt{2(1+\epsilon)\beta_k}$ and require the conditions
\begin{equation}\label{inicio2}
\frac{1}{\alpha_k\EE F(\sigma_k)^2}\sum_{n=1}^{N_1}\frac{1}{n^{2\sigma_k}}\leq\frac{1}{\log k},\text{~~~~with~~~~}\frac{\alpha_k}{\log\log\EE F(\sigma_k)^2}\to 0,
\end{equation}
and
\begin{equation}\label{cauda2}
\frac{1}{\beta_k\EE F(\sigma_k)^2}\sum_{n=N_2+1}^{\infty}\frac{1}{n^{2\sigma_k}}\leq\frac{1}{\log k},\text{~~~~with~~~~}\frac{\beta_k}{\log\log\EE F(\sigma_k)^2}\to 0.
\end{equation}

Let us consider, for $\delta>0$, the sequence $\sigma_k\to\frac{1}{2}^+$ to be  \begin{equation}\label{sigmak}
\sigma_k=\frac{1}{2}+\frac{1}{2\exp\left(k^{1+\delta}\right)}.
\end{equation}
Then, using Lemma \ref{lemma riemann sums}, the conditions \eqref{inicio2} and \eqref{cauda2} will hold if we require
\begin{equation}\label{inicio3}
N_1(\sigma_k)\leq\left(1+\exp(-k^{1+\delta})-\frac{\alpha_k}{\log k}\right)^{-\exp(k^{1+\delta})},\text{~~~~with~~~~}\frac{\alpha_k}{\log k}\to 0,
\end{equation} 
and
\begin{equation}\label{cauda3}N_2(\sigma_k)\geq \left(\frac{\log k}{\beta_k}\right)^{\exp(k^{1+\delta})},\text{~~~~with~~~~}\frac{\beta_k}{\log k}\to 0.
\end{equation}

Let us choose $N_1$ and $N_2$ assuming equality in \eqref{inicio3} and \eqref{cauda3}, and $\alpha_k=\sqrt{\log k}$. Recall that we are also looking for $N_1$ and $N_2$ satisfying \eqref{meio}, and, for that, we should have
\begin{equation*}
\beta_k\geq \log k\left(1+\exp(-(k+1)^{1+\delta})-(\log (k+1))^{-1/2}\right)^{\exp\left((k+1)^{1+\delta}-k^{1+\delta}\right)}.    
\end{equation*}
Such choice of $\beta_k$ will be possible if
\begin{equation*}
    \lim_{k\to\infty}\left(1+\exp(-(k+1)^{1+\delta})-(\log (k+1))^{-1/2}\right)^{\exp\left((k+1)^{1+\delta}-k^{1+\delta}\right)}=0,
\end{equation*}
which can be checked to be true by using L'Hôpital rule. For this limit, the necessity of the condition $\delta>0$ is crucial.

We have just found sequences $\sigma_k$, $N_1(\sigma_k)$ and $N_2(\sigma_k)$ satisfying \eqref{inicio}, \eqref{cauda} and \eqref{meio}. To complete the proof of \eqref{step1} we need to show that 
\begin{equation*}
    \PP\left(\limsup_{k\to\infty}F_2(\sigma_k)\geq(1-\gamma)\sqrt{2\log\log\EE F(\sigma_k)^2}\right)=1.
\end{equation*}
Since $N_1(\sigma_{k+1})\geq N_2(\sigma_k)$, we have the required independence needed for the second Borel-Cantelli lemma. Therefore, we must prove that the series 
\begin{equation}\label{diverge}
\sum_{k=1}^\infty\PP\left(F_2(\sigma_k)\geq(1-\gamma)\sqrt{2\log\log\EE F(\sigma_k)^2})\right)
\end{equation}
diverges.

The next paragraphs will be devoted to find a lower estimate for the probability in \eqref{diverge}. Let us recall from \eqref{barF} that $\bar{F}(\sigma)$ denotes  the normalized version of $F(\sigma)$. Since the terms $F_1(\sigma_k)$ and $F_3(\sigma_k)$ are irrelevant owing to the $(2\log\log\EE F(\sigma_k)^2)^{-1/2}$ term, we will use a lower bound as the one stated in the following:

\begin{lemma}\label{lowerbound}
	Let $f(\sigma)$ be a function that goes to $+\infty$ as $\sigma\to\frac{1}{2}^+$ and satisfies the condition
	\begin{equation}\label{condition_f}
	\lim_{\sigma\to\frac{1}{2}^+}\frac{f(\sigma)}{\sqrt{\EE F(\sigma)^2}}=0.
	\end{equation}
	Then, for all $\delta,\lambda,\epsilon>0$, there exists $\delta_1>0$ such that for $\sigma-\frac{1}{2}\leq\delta_1$, we have
\begin{equation*}
    	\PP (\bar F(\sigma)\geq \delta f(\sigma))\geq \left(\frac{1}{2}-\epsilon\right)\exp\left(-\frac{1}{2}\delta^2(1+\lambda)^2f(\sigma)^2\right).
\end{equation*}
	\end{lemma}
 The bound in Lemma \ref{lowerbound} will be used for the law of the iterated logarithm 
 with the function $f(\sigma)=\sqrt{2\log\log \EE F(\sigma)^2}$.

\begin{proof} For all $\bar\lambda>0$, let us consider the event $A=A(\sigma,\delta,\bar\lambda)$ in which $\bar F(\sigma)\in[\delta f(\sigma), \delta (1+\bar\lambda)f(\sigma)]$. Then 
	\begin{equation*}
	    \PP(\bar F(\sigma)\geq\delta f(\sigma))\geq\PP(A).
	\end{equation*}
	For each $n$, define the probability measure
	\begin{equation*}
	    \widetilde \PP_{t_0}(dX_n)=\frac{\exp\left(\frac{t_0X_n}{n^\sigma\sqrt{\EE F(\sigma)^2}}\right)}{\cosh\left(\frac{t_0}{n^\sigma\sqrt{\EE F(\sigma)^2}}\right)}\PP (dX_n),
	\end{equation*}
	where $t_0>0$ will be chosen later. The introduction of this Radon-Nikodym factor is a classical tool in the proof of the lower bound in large deviation theory. Let $\widetilde\PP(dX)$ be the probability measure consisted in the product measure of each  $\widetilde\PP(dX_n)$, $n\geq 1$.
	
	 We have
	\begin{align*}
	\PP &(\bar F(\sigma)\geq \delta f(\sigma))\geq \int_A\PP(dX)\\
	&=\exp\left(\sum_{n=1}^\infty\log\cosh\frac{t_0}{n^\sigma\sqrt{\EE F(\sigma)^2}}\right)\int_A\exp\left(-\sum_{n=1}^\infty\frac{t_0X_n}{n^\sigma\sqrt{\EE F(\sigma)^2}}\right)\widetilde\PP_{t_0} (dX).
	\end{align*}
	
	Since, on the event $A$, 
	\begin{equation*}
	    -\sum_{n=1}^\infty\frac{t_0X_n}{n^\sigma\sqrt{\EE F(\sigma)^2}}\geq -t_0\delta(1+\bar\lambda)f(\sigma),
	\end{equation*} 
	we obtain
	\begin{equation}\label{cota}
	\PP (\bar F(\sigma)\geq \delta f(\sigma))\geq\exp\left(-t_0\delta(1+\bar\lambda)f(\sigma)+\sum_{n=1}^\infty\log\cosh\frac{t_0}{n^\sigma\sqrt{\EE F(\sigma)^2}}\right)\widetilde\PP _{t_0}(A).
	\end{equation}
	
	Let us denote by $h(t)$ the function
	\begin{equation}\label{h(t)}
	h(t)=\sum_{n=1}^\infty\tanh\left(\frac{t}{n^\sigma\sqrt{\EE F(\sigma)^2}}\right)\frac{1}{n^\sigma\sqrt{\EE F(\sigma)^2}}.
	\end{equation}
	Observe $h(t)$ is an increasing function of $t$. We chose $t_0$ as the (unique) solution of the equation
	\begin{equation}\label{t0}
	\delta f(\sigma)=h(t_0).
	\end{equation}
	
	The following lemma states some properties of $t_0$. The proof will be postponed to the end of this subsection.
	\begin{lemma}\label{properties_t0}
		If $t_0$ is the solution of \eqref{t0}, then for any $\lambda>0$, there exists a $\delta_1>0$ such that, if $\sigma-\frac{1}{2}\leq\delta_1$, we have
		\begin{equation}\label{boundt01}
		\delta f(\sigma)\leq t_0\leq \delta(1+\lambda)f(\sigma).
		\end{equation}
		Moreover, 
		\begin{equation}\label{boundt02}
		-t_0\delta(1+\lambda)f(\sigma)+\sum_{n=1}^\infty\log\cosh\frac{t_0}{n^\sigma\sqrt{\EE F(\sigma)^2}}\geq -\frac{1}{2}\delta^2(1+2\lambda)^2f(\sigma)^2.
		\end{equation}
	\end{lemma} 

Using Lemma \ref{properties_t0} in \eqref{cota} with $\bar\lambda=\lambda/2$, in order to conclude the proof of Lemma \ref{lowerbound} we only need to show that for all $\epsilon>0$, exists $\delta_1>0$ such that for $\sigma-\frac{1}{2}\leq\delta_1$, we have
\begin{equation*}
    \widetilde\PP _{t_0}(\bar F(\sigma)\in [\delta f(\sigma),\delta (1+\bar\lambda)f(\sigma)])\geq \frac{1}{2}-\epsilon.
\end{equation*}
It is sufficient to prove
\begin{equation}\label{Ptilde1}
1-\widetilde\PP_{t_0}(\bar F(\sigma)<\delta f(\sigma))\geq \frac{1}{2}-\frac{\epsilon}{2}
\end{equation}
and
\begin{equation}\label{Ptilde2}
\widetilde\PP_{t_0}(\bar F(\sigma)>\delta(1+\bar\lambda)f(\sigma))\leq\frac{\epsilon}{2}.
\end{equation}

We will show that $\bar F(\sigma)-\delta f(\sigma)$ converge in law, under $\widetilde\PP_{t_0}$ to a standard Gaussian random variable, as $\sigma\to\frac{1}{2}^+$. For that, we will prove the convergence of the corresponding moment generating functions.

Observing that $M_n=\exp\left(t\sum_{k=1}^n\frac{X_k}{k^\sigma\sqrt{\EE F(\sigma)^2}}\right)\Big/b_n$, where $b_n=\prod_{k=1}^n\frac{\cosh\frac{t+t_0}{k^\sigma\sqrt{\EE F(\sigma)^2}}}{\cosh\frac{t_0}{k^\sigma\sqrt{\EE F(\sigma)^2}}}$, is a martingale under $\widetilde{\PP}_{t_0}$, with respect to the $\sigma$-algebra $\mathcal{F}_n$ generated by $\left\{X_1,\ldots,X_n\right\}$, we can reproduce Lemma \ref{infiteproduct} for $\widetilde{\EE}_{t_0}$. Then  
\begin{equation}\label{expectation_tilde}
\widetilde{\EE}_{t_0}\left[e^{t\bar F(\sigma)}\right]=\frac{\exp\left(\sum_{n=1}^\infty\log\cosh\frac{t+t_0}{n^\sigma\sqrt{\EE F(\sigma)^2}}\right)}{\exp\left(\sum_{n=1}^\infty\log\cosh\frac{t_0}{n^\sigma\sqrt{\EE F(\sigma)^2}}\right)}=\frac{\EE \left[e^{(t+t_0)\bar F(\sigma)}\right]}{\EE \left[e^{t_0\bar F(\sigma)}\right]}.
\end{equation}

Using the estimates
 \begin{equation}\label{logcosh}
 \frac{x^2}{2}-\frac{x^4}{8}\leq\log\cosh x\leq\frac{x^2}{2}, \text{ for all } x\in\mathbb R,
 \end{equation}  
 \begin{equation}\label{somaquadratica}
 \sum_{n=1}^\infty\frac{1}{n^{4\sigma}}\leq\sum_{n=1}^\infty\frac{1}{n^2}=\frac{\pi^2}{6}
 \end{equation}
 and Lemma \ref{infiteproduct}, we have 
\begin{equation}\label{moment function}
\exp\left(\frac{t^2}{2}-\frac{t^4\pi^2}{48(\EE F(\sigma)^2)^2}\right)\leq\EE\left[e^{t\bar F(\sigma)}\right]\leq\exp\left(\frac{t^2}{2}\right).
\end{equation}

Note that, in particular, \eqref{moment function} gives us $\lim_{\sigma\to\frac{1}{2}^+}\EE\left[e^{t\bar F(\sigma)}\right]=e^{\frac{t^2}{2}}$,
which yields an alternative proof of the Central Limit Theorem for $\bar{F}(\sigma)$ that was proved in \cite{aymonerealzerosofdirichletseries} using the convergence of characteristic functions.
 
Using \eqref{moment function} in \eqref{expectation_tilde} we provide the following upper and lower bounds:
\begin{equation*}
    \widetilde\EE_{t_0}\left[e^{t\bar F(\sigma)}\right]e^{-t\delta f(\sigma)}\leq\exp\left(\frac{t^2}{2}+t(t_0-\delta f(\sigma))+\frac{t_0^4\pi^2}{48(\EE F(\sigma)^2)^2}\right)
\end{equation*}
and
\begin{equation*}
    \widetilde\EE_{t_0}\left[e^{t\bar F(\sigma)}\right]e^{-t\delta f(\sigma)}\geq\exp\left(\frac{t^2}{2}+t(t_0-\delta f(\sigma))-\frac{(t+t_0)^4\pi^2}{48(\EE F(\sigma)^2)^2}\right).
\end{equation*}

Thus, by \eqref{boundt01} and the condition \eqref{condition_f}, we obtain
\begin{equation*}
    \lim_{\sigma\to\frac{1}{2}^+}\widetilde\EE_{t_0}\left[e^{t\bar F(\sigma)}\right]\cdot e^{-t\delta f(\sigma)}=e^{\frac{t^2}{2}},
\end{equation*}
which gives us the convergence (under the law of $\tilde\PP_{t_0}$) to the standard Gaussian variable, therefore
\begin{equation*}
    \lim_{\sigma\to\frac{1}{2}^+}\widetilde\PP_{t_0}(\bar F(\sigma)-\delta f(\sigma)>0)=\frac{1}{2},
\end{equation*}
which proves \eqref{Ptilde1}.

It remains to prove \eqref{Ptilde2}. Since $f(\sigma)$ explodes as $\sigma\to\frac{1}{2}^+$, for any fixed $a>0$, we have
\begin{align*}
\limsup_{\sigma\to\frac{1}{2}^+}&~\widetilde\PP_{t_0}(\bar F(\sigma)>\delta(1+\bar\lambda)f(\sigma))\\
&=\limsup_{\sigma\to\frac{1}{2}^+}\widetilde\PP_{t_0}(\bar F(\sigma)-\delta f(\sigma)>\delta\bar\lambda f(\sigma))\\
&\leq \limsup_{\sigma\to\frac{1}{2}^+}\widetilde\PP_{t_0}(\bar F(\sigma)-\delta f(\sigma)>a)\\
&=\int_a^\infty \frac{e^{-x^2/2}}{\sqrt{2\pi}}dx, 
\end{align*} 
which goes to zero as $a\to\infty$. 

This shows that $\widetilde\PP_{t_0}(\bar F(\sigma)>\delta(1+\bar\lambda)f(\sigma))$ can be arbitrarily small as $\sigma\to\frac{1}{2}^+$, which gives us \eqref{Ptilde2}. And this completes the proof of Lemma \ref{lowerbound}.
\end{proof}

\begin{proof}[Proof of Lemma \ref{properties_t0}]
Using that $\tanh x\leq x$, for $x\geq 0$, we obtain
\begin{equation*}
\delta f(\sigma)=h(t_0)\leq \frac{t_0}{\EE F(\sigma)^2}\sum_{n=1}^\infty\frac{1}{n^{2\sigma}}=t_0,    
\end{equation*}
which proves the lower bound of $t_0$ stated in \eqref{boundt01}.

For the upper bound we use that $\tanh x\geq x-\frac{x^3}{2}$ and \eqref{somaquadratica}. We obtain
\begin{equation*}
\delta f(\sigma)\geq t_0-t_0^3\frac{\pi^2}{12(\EE F(\sigma)^2)^2}=:g(t_0).
\end{equation*}
The cubic function $g(t)$ hits its maximum at $\hat t=\frac{2}{\pi}\EE F(\sigma)^2$, and $g(\hat t)=\frac{2}{3}\hat t$.

Now, since $f$ satisfies \eqref{condition_f}, we have, for $\sigma$ close enough to $\frac{1}{2}$, that 
\begin{equation*}
    \delta f(\sigma)\leq \epsilon\sqrt{\EE F(\sigma)^2}<\frac{2}{3}\hat t.
\end{equation*}
Then, since the increasing function $h$ satisfies $g(t)\leq h(t)\leq t$, the solution $t_0$ of \eqref{t0} must satisfies $t_0<\hat t$, which implies $g(t_0)\geq \frac{2}{3}t_0$.

This implies $t_0\leq \frac{3}{2}g(t_0)\leq\frac{3}{2}\delta f(\sigma)$. Notice that this already gives us an upper bound for $t_0$, however, this bound can be improved. Indeed,
\begin{equation*}
    t_0\leq \delta f(\sigma)+t_0^3\frac{\pi^2}{12(\EE F(\sigma)^2)^2}\leq\delta f(\sigma)\left(1+\frac{27\delta^2\pi^2}{96}\frac{f(\sigma)^2}{(\EE F(\sigma)^2)^2}\right).
\end{equation*}
Again, by \eqref{condition_f}, there exists $\delta_1>0$ such that, if $\sigma-\frac{1}{2}\leq \delta_1$, we have the upper bound stated in \eqref{boundt01}.

Now we will prove \eqref{boundt02}. Using \eqref{logcosh}, \eqref{somaquadratica} and \eqref{boundt01}, we obtain
\begin{align}
-t_0\delta(1+\lambda)&f(\sigma)+\sum_{n=1}^\infty\log\cosh\frac{t_0}{n^\sigma\sqrt{\EE F(\sigma)^2}}\nonumber\\
&\geq -t_0\delta(1+\lambda)f(\sigma)+\frac{t_0^2}{2}-\frac{\pi^2}{48}\frac{t_0^4}{(\EE F(\sigma)^2)^2}\nonumber\\
&\geq -\frac{1}{2}\delta^2f(\sigma)^2\left(2(1+\lambda)^2-1+\frac{\pi^2\delta^2(1+\lambda)^4}{24}\frac{f(\sigma)^2}{(\EE F(\sigma)^2)^2}\right).\label{revisao1}
\end{align}
Using \eqref{condition_f} again, we have that, for $\sigma$ close to $\frac{1}{2}^+$, 
\begin{equation*}
\left(\frac{f(\sigma)}{\EE F(\sigma)^2}\right)^2\leq \frac{24}{\pi^2\delta^2(1+\lambda)^4}\cdot 2\lambda^2.
\end{equation*}
Then, the expression in \eqref{revisao1} is bounded below by $-\frac{1}{2}\delta^2f(\sigma)^2(1+2\lambda)^2$, which proves \eqref{boundt02}.
\end{proof}

Now, by Lemma \ref{lowerbound}, and considering $\sigma_k$ as in \eqref{sigmak}, we have that, if $k$ is big enough, 
\begin{equation*}
    \PP\left(F_2(\sigma_k)\geq(1-\gamma)\sqrt{2\log\log \EE F(\sigma_k)^2}\right)\geq\left(\frac{1}{2}-\epsilon\right)\frac{1}{k^{(1-\gamma)^2(1+\lambda)^2(1+\delta)}}.
\end{equation*}
Therefore, for any $\gamma>0$, a suitable choice of the parameters $\lambda,\delta$ gives us the divergence of the series \eqref{diverge}. Thus, the proof of \textit{step 1} is completed.

\subsection*{Step 2}
\begin{lemma}\label{lemma controle superior sobre subsequencia} Let $\epsilon>0$ be small and $\delta=\epsilon/2$. Let $\sigma_k=\frac{1}{2}+\frac{1}{2\exp(k^{1-\delta})}$. Then it \textit{a.s.} holds that 
	\begin{equation*}
	\limsup_{k\to\infty} \frac{\bar F(\sigma_k)}{\sqrt{2\log\log \EE F(\sigma_k)^2}}\leq \sqrt{1+\epsilon}.
	\end{equation*}
\end{lemma}
\begin{proof}
	We have, by the Hoeffding inequality that 
	\begin{align*}
	\PP\left(\bar F(\sigma_k)\geq \sqrt{2(1+\epsilon)\log\log \EE F(\sigma_k)^2}\right)\leq \exp\left(-(1+\epsilon)\log\log\EE F(\sigma_k)^2\right).
	\end{align*}
	By Lemma \ref{lemma riemann sums}, we have
	$\log\log \EE F(\sigma_k)^2\geq \log\log \frac{1}{2\sigma_k-1}=(1-\delta)\log k$. We also have $(1+\epsilon)(1-\delta)=1+\gamma$, where 
	$\gamma=\epsilon/2-\epsilon^2/2>0$, provided that $\epsilon>0$ is small. Thus
	\begin{align*}
	\PP\left(\bar F(\sigma_k)\geq \sqrt{2(1+\epsilon)\log\log \EE F(\sigma_k)^2}\right)\leq \exp(-(1+\gamma)\log k)=\frac{1}{k^{1+\gamma}}.
	\end{align*}
	Hence, 
	\begin{align*}
	\sum_{k=1}^\infty\PP(F(\sigma_k)\geq \sqrt{2(1+\epsilon)\EE F(\sigma_k)^2\log\log \EE F(\sigma_k)^2})<\infty.
	\end{align*}
	The Borel-Cantelli Lemma completes the proof. \end{proof}

\subsection*{Step 3}
\begin{lemma}\label{lemma controle entre os intervalos} Let $\sigma_k=\frac{1}{2}+\frac{1}{2\exp(k^{1-\delta})}$, where $\delta>0$ is a fixed small constant. For $\PP$ almost all $\omega\in\Omega$, there exists $k_0=k_0(\omega)$ such that for $k\geq k_0$, we have that
	\begin{equation*}
	\max_{\sigma\in[\sigma_k,\sigma_{k-1}]}|F(\sigma)-F(\sigma_k)|\ll \sqrt{\EE F (\sigma_k)^2}.
	\end{equation*}
\end{lemma}
\begin{proof}
	For a non negative integer $l$, we define $\tau_{l,0}=\tau_{l,0}(k)=\sigma_k$ and $\tau_{l,n}=\tau_{l,n}(k)=\sigma_k+\frac{n}{2^l}(\sigma_{k-1}-\sigma_k)$, where $0\leq n\leq 2^l$. Let $\lambda_{k,l}$ be a constant to be chosen later and consider the event
	\begin{equation*}
	\mathcal{A}_{l,k}=[\max_{0\leq n\leq 2^l-1}|F(\tau_{l,n+1})-F(\tau_{l,n})|\geq \lambda_{k,l}].
	\end{equation*}
	Let 
	\begin{equation*}
	U_k(\omega)=\min\bigg{\{}u\in\NN:\omega\in \bigcap_{l=u}^\infty \mathcal{A}_{l,k}^c \bigg{\}}.
	\end{equation*}
	One can check that $[U_k\leq L]=\bigcap_{l=L}^\infty \mathcal{A}_{l,k}^c$. Thus
	\begin{align*}
	\PP(U_k>L)\leq \sum_{l=L}^\infty \PP(\mathcal{A}_{l,k})\leq \sum_{l=L}^\infty \sum_{n=0}^{2^l-1}\PP(|F(\tau_{l,n+1})-F(\tau_{l,n})|\geq \lambda_{k,l}).
	\end{align*}
	Next, we will estimate each probability in the inner sum above. We have, by the mean value theorem, that
	\begin{align*}
	F(\tau_{l,n+1})-F(\tau_{l,n})=(\tau_{l,n+1}-\tau_{l,n})\sum_{m=1}^\infty -X_mm^{-\theta_{l,n,m}}\log m,  
	\end{align*} 
	where $\theta_{l,n,m}\in(\tau_{l,n}, \tau_{l,n+1})$. Thus, 
	\begin{align*}
	\EE |F(\tau_{l,n+1})-F(\tau_{l,n})|^2&=(\tau_{l,n+1}-\tau_{l,n})^2\sum_{m=1}^\infty m^{-2\theta_{l,n,m}}\log^2 m\\
	&\leq \frac{(\sigma_{k}-\sigma_{k-1})^2}{4^l}\sum_{m=1}^\infty m^{-2\sigma_k}\log^2 m\\
	&\ll \frac{(\sigma_{k}-\sigma_{k-1})^2}{4^l}\int_1^\infty t^{-2\sigma_k}\log^2 t dt\\
	&= \frac{(\sigma_{k}-\sigma_{k-1})^2}{4^l}\int_0^\infty t^2 \exp(-(2\sigma_k-1)t)dt\\
	&=  \frac{(\sigma_{k}-\sigma_{k-1})^2}{4^l}\frac{2}{(2\sigma_k-1)^3}\\
	&\ll \frac{(\exp(-k^{1-\delta})/k^\delta)^2}{4^l}\frac{1}{(\exp(-k^{1-\delta}))^3}\\
	&= \frac{\exp(k^{1-\delta})}{4^lk^{2\delta}}.\\
	\end{align*}
	Thus, by the Hoeffding inequality, for some constant $c_0$, we have that
	\begin{align*}
	\PP(|F(\tau_{l,n+1})-F(\tau_{l,n})|\geq \lambda_{k,l})\leq \exp\bigg{(}-c_0\frac{\lambda_{k,l}^2}{2}\frac{4^lk^{2\delta}}{\exp(k^{1-\delta})}  \bigg{)},
	\end{align*}
	and hence
	\begin{align*}
	\PP(U_k>1)&\leq \sum_{l=1}^\infty 2^l\exp\bigg{(}-\frac{c_0\lambda_{k,l}^2}{2}\frac{4^lk^{2\delta}}{\exp(k^{1-\delta})}\bigg{)}\\
	&=\sum_{l=1}^\infty\exp\bigg{(}-\frac{c_0\lambda_{k,l}^2}{2}\frac{4^lk^{2\delta}}{\exp(k^{1-\delta})}+l\log 2\bigg{)}.
	\end{align*}
	Choose 
	\begin{equation*}
	\lambda_{k,l}^2=\frac{2}{c_0}\frac{\exp(k^{1-\delta})}{4^l}l.
	\end{equation*}
	Hence,
	\begin{align*}
	\PP(U_k>1)&\leq\sum_{l=1}^\infty\exp((-k^{2\delta}+\log 2)l)\ll \exp(-k^{2\delta}).
	\end{align*}
	Thus, $\sum_{k=1}^\infty \PP (U_k>1)<\infty$, and hence, by the Borel-Cantelli Lemma, there exists a set $\Omega^*$ of probability $1$ such that for all $\omega\in\Omega^*$, $U_k(\omega)=1$, for $k\geq k_0(\omega)$. 
	
	Let $ D_{l,k}=\{\tau_{l,n}:0\leq n\leq 2^l\}$ and put $D_k=\bigcup_{l=0}^\infty D_{l,k} $. We shall fix $\omega\in\Omega^*$ and $m\geq n\geq 1$ where $k\geq k_0(\omega)$ and show that for $0<|s-t|<\frac{|\sigma_k-\sigma_{k-1}|}{2^n}$, $|F(s)-F(t)|\leq 2\sqrt{\frac{2}{c_0}} \exp(k^{1-\delta}/2)\sum_{j=n+1}^m\frac{\sqrt{j}}{2^{j}}$, for all $t,s\in D_{m,k}$. Indeed, for $m=n+1$, we can only have that $|s-t|=|\tau_{m,n}-\tau_{m,n+1}|$, and hence $|F(s)-F(t)|\leq \lambda_{k,m}= \sqrt{\frac{2}{c_0}} \exp(k^{1-\delta}/2)\frac{\sqrt{m}}{2^{ m}}$. Suppose now that the claim is true for $m=n+1,...,M-1$ and consider $m=M$.
	Let $s,t\in D_{M,k}$ with $0<|s-t|<\frac{|\sigma_k-\sigma_{k-1}|}{2^n}$. Consider $t'=\max\{u\leq t: u\in D_{M-1,k}\}$ and $s'=\min\{u\geq s: u\in D_{M-1,k}\}$. Thus
	\begin{align*}
	|F(s)-F(t)|&\leq |F(s)-F(s')|+|F(t)-F(t')|+|F(s')-F(t')|\\
	&\leq 2\lambda_{k,M}+2\sqrt{\frac{2}{c_0}}\exp(k^{1-\delta}/2)\sum_{j=n+1}^{M-1}\frac{\sqrt{j}}{2^{j}}\\
	&=2\sqrt{\frac{2}{c_0}}\exp(k^{1-\delta}/2)\sum_{j=n+1}^{M}\frac{\sqrt{j}}{2^{j}}.
	\end{align*}
	Now, for any $s,t\in D_k$ with $|s-t|\leq \frac{|\sigma_k-\sigma_{k-1}|}{2}$, select $n$ such that $\frac{|\sigma_k-\sigma_{k-1}|}{2^{n+1}}\leq|s-t|<\frac{|\sigma_k-\sigma_{k-1}|}{2^n}$. Thus 
	\begin{equation*}
	|F(s)-F(t)|\leq 2\sqrt{\frac{2}{c_0}}\exp(k^{1-\delta}/2)\sum_{j=n+1}^{\infty}\frac{\sqrt{j}}{2^{j}}\ll \exp(k^{1-\delta}/2).
	\end{equation*} 
	As $D_k$ is dense in the interval $[\sigma_k,\sigma_{k-1}]$ and $F$ is analytic, in particular it is continuous, we conclude that $|F(s)-F(t)|\ll \exp(k^{1-\delta}/2)$, for all $s,t\in[\sigma_k,\sigma_{k-1}]$ with $|s-t|\leq \frac{|\sigma_k-\sigma_{k-1}|}{2}$. Finally, observe that $\exp(k^{1-\delta}/2)=\frac{1}{\sqrt{2\sigma_k-1}}$, and that for $\sigma\in[\sigma_k, \sigma_{k-1}]$, $|\sigma-(\sigma+\sigma_k)/2|=|\sigma_k-(\sigma+\sigma_k)/2|\leq \frac{|\sigma_k-\sigma_{k-1}|}{2}$, and hence
	\begin{align*}
	|F(\sigma)-F(\sigma_k)|&\leq |F(\sigma)-F((\sigma_k+\sigma)/2)|+|F(\sigma_k)-F((\sigma_k+\sigma)/2)|\\
	&\ll \frac{1}{\sqrt{2\sigma_k-1}}.
	\end{align*}
	Since $\frac{1}{2\sigma_k-1}\leq\EE F(\sigma_k)^2$ (see Lemma \ref{lemma riemann sums}), the proof is completed.
\end{proof}

\subsection*{Step 4} 
\begin{lemma}\label{lema grand finale}
We have that
\begin{equation*}
\PP \left(\limsup_{\sigma\to 1/2^+}\frac{\bar{F}(\sigma)}{\sqrt{2\log\log \EE F(\sigma)^2}}\leq 1+\gamma\right)=1.
\end{equation*}
\end{lemma}
\begin{proof}
Let $k_0=k_0(\omega)$ be as in Lemma \ref{lemma controle entre os intervalos}, and $1/2<\sigma<\sigma_{k_0}$. By Lemma \ref{lemma riemann sums}, we have that for all $k$, holds 
\begin{equation}\label{riemann sigmak}
\exp\left(k^{1-\delta}\right)\leq\EE F(\sigma_k)^2\leq 1+\exp\left(k^{1-\delta}\right). 
\end{equation}

Lets us assume that $\sigma\in[\sigma_k, \sigma_{k-1}]$ and write
\begin{equation*}
    \frac{\bar F(\sigma)}{\sqrt{2\log\log \EE F(\sigma)^2}}=\frac{F(\sigma_k)}{\sqrt{2\EE F(\sigma)^2\log\log \EE F(\sigma)^2}}+\frac{F(\sigma)-F(\sigma_k)}{\sqrt{2\EE F(\sigma)^2\log\log \EE F(\sigma)^2}}.
\end{equation*}
By Lemma \ref{lemma controle superior sobre subsequencia} and \eqref{riemann sigmak}, we have
\begin{align*}
\frac{F(\sigma_k)}{\sqrt{2\EE F(\sigma)^2\log\log \EE F(\sigma)^2}}& \leq \sqrt{1+\epsilon}\frac{\sqrt{\EE F(\sigma_k)^2\log\log \EE F(\sigma_k)^2}}{\sqrt{\EE F(\sigma_{k-1})^2\log\log \EE F(\sigma_{k-1})^2}}\\
&\leq\sqrt{1+\epsilon}\left(1+r_\delta(k)\right),
\end{align*}
for a function $r_\delta(k)$ satisfying $\lim_{k\to\infty}r_\delta(k)=0$.

Now, by Lemma \ref{lemma controle entre os intervalos}, and using again \eqref{riemann sigmak}, we have that there exists a constant $c_0$ that does not depend on $k$ such that
\begin{align*}
\frac{F(\sigma)-F(\sigma_k)}{\sqrt{2\EE F(\sigma)^2\log\log \EE F(\sigma)^2}}&\leq \frac{c_0\sqrt{\EE F(\sigma_{k})^2}}{\sqrt{\EE F(\sigma_{k-1})^2\log\log \EE F(\sigma_{k-1})^2}}\\
&\leq s_\delta(k), 
\end{align*}
for a function $s_\delta(k)$ satisfying $\lim_{k\to\infty}s_\delta(k)=0$.

Sending $k\to\infty$ we conclude the proof of Lemma \ref{lema grand finale}.
\end{proof}

\noindent \textbf{Acknowledgements.} We would like to thank the anonymous referees for a careful reading of the paper and for useful suggestions and corrections.

\newpage

{\small{\sc \noindent Marco Aymone \\
Departamento de Matem\'atica, Universidade Federal de Minas Gerais, Av. Ant\^onio Carlos, 6627, CEP 31270-901, Belo Horizonte, MG, Brazil.} \\
\textit{Email address:} marco@mat.ufmg.br}
\vspace{0.5cm}

{\small{\sc \noindent Susana Fr\'ometa \\
Departamento de Matem\'atica, Universidade Federal do Rio Grande do Sul, Av. Bento Gon\c calves, 9500, CEP 91509-900, Porto Alegre, RS, Brazil.} \\
\textit{Email address:} susana.frometa@ufrgs.br}
\vspace{0.5cm}

{\small{\sc \noindent Ricardo Misturini \\
Departamento de Matem\'atica, Universidade Federal do Rio Grande do Sul, Av. Bento Gon\c calves, 9500, CEP 91509-900, Porto Alegre, RS, Brazil.} \\
\textit{Email address:} ricardo.misturini@ufrgs.br}
\vspace{0.5cm}

\end{document}